%% file: Lefschetz.tex
\documentclass[11pt]{article}
 
\usepackage[hypertex]{hyperref}
\usepackage[active]{srcltx}

\usepackage{epsfig}
\usepackage{graphicx}

\usepackage[active]{srcltx}

\usepackage{amssymb}
\usepackage{amscd}
\usepackage{amsmath}
\usepackage{amsfonts}
\usepackage{amsthm}
\usepackage{mathrsfs}

\usepackage{color}

\newtheorem{theo}{Theorem}[subsection]
\newtheorem{prop}[theo]{Proposition}
\newtheorem{lem}[theo]{Lemma}

\theoremstyle{remark}

\newcommand{\R}{{\mathbb{R}}}

\newcommand{\C}{{\mathbb{C}}}
\newcommand{\Z}{{\mathbb{Z}}}

\newcommand{\al}{\alpha} 
 
\newcommand{\Om}{\Omega} 
\newcommand{\te}{\theta} 

\newcommand{\om}{\omega} 
 
\newcommand{\si}{\sigma} 

\newcommand{\Ci}{{\mathcal{C}}^{\infty}}

\newcommand{\id}{\operatorname{id}} 
\newcommand{\preq}{{\mathcal{P}}}  

\newcommand{\Hom}{\operatorname{Hom}}

\newcommand{\der}{{\mathcal{L}}}  

\newcommand{\Hilb}{{\mathcal{H}}}   
\newcommand{\Fourier}{{\mathcal{F}}} 

\newcommand{\Sl}{\operatorname{Sl}}

\renewcommand{\top}{\operatorname{top}}
\newcommand{\demi}{{\mathcal{D}}}
\newcommand{\dcomp}{\circ_{\demi}}

\newcommand{\sym}{\operatorname{Sym}} 
\newcommand{\Sp}{\operatorname{Sp}}
\newcommand{\Mp}{\operatorname{Mp}}
\newcommand{\Gl}{\operatorname{Gl}}
\newcommand{\Ml}{\operatorname{Ml}}

\title{A Lefschetz fixed point formula for symplectomorphisms} 

\author{Laurent CHARLES \footnote{Institut de
    Math{\'e}matiques de Jussieu (UMR 7586), Universit{\'e} Pierre et
    Marie Curie -- Paris 6, Paris, F-75005 France.}} 

\begin{document}

\maketitle

\bibliographystyle{plain}


\begin{abstract} 
Consider a compact K{\"a}hler manifold endowed with a prequantum
bundle. Following the geometric quantization scheme, the associated
quantum spaces are the spaces of holomorphic sections of the tensor
powers of the prequantum
bundle. 
In this paper we construct an asymptotic representation of the prequantum bundle
automorphism group in these quantum spaces. We estimate the
characters of these representations under some transversality assumption. The formula obtained generalizes in some
sense the Lefschetz fixed point formula for the automorphisms of the
prequantum bundle preserving its holomorphic structure. Our results
 will be applied in two forthcoming papers to the quantum representation
 of the mapping class group.   
\end{abstract}

Consider a compact K{\"a}hler manifold $M$ endowed with a Hermitian
holomorphic bundle $L \rightarrow M$ whose curvature  is the
fundamental two-form. In the point of view of geometric quantization,
$M$ is the classical
phase space and the space $H^0 (M, L)$ of holomorphic sections of $L$
is the quantum
space. 

The group of holomorphic automorphisms of $L$ acts naturally on the
quantum space. Furthermore, if the higher cohomology groups ($H^q(M,
L),$ $q \geqslant 1$)  of the sheaf of holomorphic sections of $L$ are
all trivial, the holomorphic Lefschetz fixed point formula expresses
the characters of this representation in
terms of characteristic classes of $M$ and $L$. 

With the physical interpretation in mind, it is natural to  consider
  the prequantum bundle automorphisms instead of the holomorphic automorphisms. These are the automorphisms
of $L$ preserving the Chern connection and the metric but not
necessarily the holomorphic structure. Whereas the group of
holomorphic automorphism is finite dimensional, the group of the
prequantum bundle automorphisms is infinite-dimensional. Its Lie
algebra identifies with the Poisson algebra of $M$. Furthermore each Hamiltonian symplectomorphism of $M$ lifts to a
prequantum bundle automorphism, and if $M$ is simply connected, each
symplectomorphism isotopic to the identity is Hamiltonian.  
The goal of this paper is to
define an asymptotic representation of these automorphisms such that
a suitable version of the Lefschetz fixed point formula holds. Here the terms
asymptotic refers to the semi-classical limit, obtained by
replacing the prequantum bundle $L$ by its $k$-th tensor power with
large value of $k$.

It is convenient to work with the metaplectic correction. Let $\delta
\rightarrow M$ be a half-form bundle, that is a square root of the
canonical bundle of $M$. Such a bundle exists if and only if the
second Stiefel-Whitney class of $M$ vanishes. We will define the notion of half-form
bundle automorphisms.  For the introduction, it is sufficient to know that any symplectomorphism
of $M$ isotopic to the identity lifts to a half-form bundle
automorphism. Furthermore on any component of $M$, this lifts is unique
up to a sign. 

Consider two automorphisms $\Phi_L$ and $\Psi$  of the prequantum bundle and the half-form bundle
respectively which lifts the same symplectomorphism $\Phi$ of $M$. Then we
will define a class $U(\Phi_L,\Psi)$ which consists of sequences 
$$ T_k : H^0(M,L^k \otimes \delta ) \rightarrow H^0 (M, L^k \otimes
\delta), \qquad k=1, 2,\ldots $$
of unitary maps whose Schwartz kernel has a precise asymptotic. Without going into
the complete definition, let us describe the main characteristics. 
\begin{itemize}
\item The Schwartz kernel concentrate on the graph
of $\Phi^{-1}$ in the sense that
$$ T_k (y, x) = O ( k^{-N}), \quad \forall N$$
for any $y$ and $x \in M$ such that  $y\neq \Phi (x)$. 

\item The asymptotic on the graph is given in terms of
$\Phi_L$ and $\Psi$ by
$$  T_k (\Phi(x), x) = \Bigl( \frac{k}{2 \pi} \Bigr)^n \Phi_L(x)^k
  \otimes ( \Psi (x)  + O(k^{-1}) ).$$
\end{itemize}
The precise definition is given in sections \ref{sec:four-integr-oper} and \ref{sec:unitary-maps}. The main properties of these operators are the following.

\begin{itemize}
\item  $U(\Phi_L,\Psi)$ is not empty. For any sequences $(T_k)$ and $(T_k')$ in
$U(\Phi_L,\Psi)$ we have that $$T'_k T_k^{-1} = \id + O(k^{-1})$$ in
uniform norm. 

\item We have an asymptotic representation in the sense that
$$ U(\Phi_L,\Psi) U(\Phi_L',\Psi') = U(\Phi_L \circ \Phi_L',\Psi \circ_{\demi}
\Psi').$$

\item When the graph of $\Phi$ intersects transversally the diagonal, we can
estimate the trace of any $(T_k)$ in $U(\Phi_L, \Psi)$:
\begin{gather} \label{eq:trace}
 \operatorname{ Tr} ( T_k) = \sum_{x = \Phi (x) } \frac{ i^ {m_x}  u_x^k } { | \det ( \id - T_x \Phi ) |^{1/2}} + O(k^{-1}) 
\end{gather}
where for any fixed point $x$ of $\Phi$, $u_x \in \C$ is the trace of
the endomorphism $\Phi_{L}(x)$ and $m_x \in \Z / 4\Z$ depends only
on $T_x \Phi$ and $\Psi (x)$. 
\end{itemize}

Precise statements are given in theorems \ref{sec:unitary-maps-comp} and \ref{sec:trace-estimates}. 
These results are consequences of more general statements where the
half-form line bundle is replaced by any auxiliary line bundle,
cf. theorem \ref{the:P1} and \ref{the:trace-estimate}. However the
quantization with metaplectic correction presented here has the following
particular features. First, we have an asymptotic
representation of a finite cover of the prequantum bundle automorphism
group, whereas with a general auxiliary line bundle we have to
consider a $U(1)$-extension. 
 Second, the formula giving the asymptotic of the trace is much more complicated for a general auxiliary line bundle. It involves the complex structure of $M$ in a essential way, whereas it depends only on the symplectic data in the half-form case.

 The index $m_x$ appearing in the estimate of the trace
 (\ref{eq:trace}) is one of the main point of this paper. It is
 similar to a Maslov or a Conley-Zehnder index. Our definition requires a choice of a complex polarization, whereas the already known definitions of indices for symplectomorphisms involved a choice of a real polarization. Besides the definition given in section \ref{sec:definition-index}, we propose a simple useful characterization in section \ref{sec:char-index}. We also compute the index of the elements of the metalinear and unitary groups. 

The results of the paper relies on the articles \cite{oim_qm}
and \cite{oim_mc}. In \cite{oim_qm} we proposed an elementary
definition of a Fourier integral operator in the context of geometric
quantization of K{\"a}hler manifolds. Previously, a definition was given by Zelditch \cite{Ze}
using the general theory of Toeplitz operators of Boutet de Monvel and
Guillemin \cite{BoGu}. The interest of half-form bundle for these
Fourier integral operators was understood in \cite{oim_mc} where the
spaces $U( \Phi_L, \Psi)$ were introduced. The estimate of the trace
with the definition of the index is new. Of course, it is very similar
to the known formula for the usual Fourier integral operators. It was one of our goal to obtain the closest formula to the usual case. 

In a sequel of this paper we will apply our result to the quantum
representations of the mapping class group defined in topological
quantum field theory. Estimating the character of these
representations with our formula, we will obtain the leading order
behavior of the quantum invariant of some 3-dimensional manifolds in
the large level limit. These asymptotics were initially obtained by
Witten in \cite{Wi} by doing the perturbative theory of some Feynman
path integral and have been rigorously proved only in few cases. The
article \cite{oim_MG} will be devoted to $\Sl (2, \Z)$ and \cite{oim_MCG} to the mapping class group of surfaces with genus greater than 2.  

The layout of this paper is as follows. Sections \ref{sec:half-forms},
\ref{sec:metaplectic-group} and \ref{sec:gener-half-form} are devoted
to linear algebra preliminaries.
In section \ref{sec:half-forms}, we define the
half-form morphisms and the symplectic linear category with
half-forms. The automorphism group of an object in this category is
a concrete realization or the metaplectic group. It is  the subject of
section \ref{sec:metaplectic-group}, where the index of some
of its element is defined. In section \ref{sec:gener-half-form}, we
generalize the previous considerations, which is necessary for the
application to the quantum representation of the mapping class
group. In section \ref{sec:quant-kahl-manif}, we consider the
quantization of K{\"a}hler manifold with an auxiliary line bundle,
introduce the operator quantizing the symplectomorphisms and estimate
their trace. In section \ref{sec:quant-half-form}, we treat the particular case where the
auxiliary bundle is a half-form bundle. In appendix
\ref{sec:linear-quantization}, we introduce a functor from the
category of symplectic vector spaces with polarization and half-form to
the category of Hilbert spaces. Applying this functor to the
automorphism group of a single object, we recover the well-known metaplectic
representation. With this elementary construction in mind, one can
view the quantization of symplectomorphisms studied in this article as a generalization of the
metaplectic representation.

\section{Half-forms} \label{sec:half-forms}

\subsection{Complex structures and canonical lines} 
Let $S$ be a symplectic real vector space. A positive
polarization $E$ of $S$ is a Lagrangian subspace of $S \otimes \C$
such that  
$$\tfrac{1}{i} \om ( x, \overline
{x}) >0$$ for any non-vanishing $x \in E$. Any positive polarization has
a canonical Hermitian
scalar product given by $$ ( x, y) \rightarrow \tfrac{1}{i} \om (x, \overline{
  y}) .$$ The set of  positive polarizations of $S$ is a
contractible topological space. 

For any  positive polarizations, we consider the canonical
line $\wedge^{\top} E^* $ with its associated Hermitian product.
%
%
%
For any positive polarizations $E_a$ and $E_b$, $S \otimes \C$
is the direct sum of $ E_a$ and $\overline{E}_{b}$. Let
$\pi_{E_b, E_a}: E_{b} \rightarrow E_{a}$ be the restriction of
the projection onto  $ E_{a}$ with kernel
$\overline{E}_{b}$. This map is invertible. Define
$$ \Psi_{E_a, E_b} = \pi_{E_b, E_a} ^* : \wedge^{\top} E^*_{a}
\rightarrow  \wedge ^{\top} E^*_{b}$$
Given three positive polarizations, let $\zeta( E_a, E_b,
E_c)$ be the complex number such that 
\begin{gather}  \label{eq:def_zeta}
 \Psi _{ E_a, E_c} = \zeta( E_a, E_b, E_c) \Psi_{ E_b, E_c} \circ
\Psi_{ E_a, E_b}
\end{gather}
Define $\zeta^{1/2}( E_a, E_b , E_c)$ to be the square root depending
continuously on $E_a$, $E_b$ and $E_c$ and taking the value 1 when
$E_a = E_b = E_c$. 

\subsection{Half-form lines} 

For any  positive polarization $E$, a half-form line of $E$ is  
a complex line $\delta$ together with an isomorphism $\varphi : \delta
^{\otimes 2} \rightarrow  \wedge ^{\top} E^*$. We endow $\delta$ with
the scalar product making $\varphi$ a unitary map.

Let us consider the
category $\demi (S)$ with objects the triples $(E, \delta,
\varphi)$ consisting of a  positive polarization together with a half-form line.
The morphisms from $(E_a , \delta_a, \varphi_a)$ to $(E_b,  \delta_b, \varphi_b)$ are the linear maps $\Psi : \delta_a \rightarrow
\delta_b$ satisfying $$ \varphi_b \circ \Psi^{\otimes 2} = \Psi_{E_a, E_b} \circ
\varphi_a.$$
The composition of a morphism $\Psi: (E_a , \delta_a, \varphi_a)
\rightarrow (E_b,  \delta_b, \varphi_b)$ with a morphism $\Psi ' : (E_b , \delta_b, \varphi_b)
\rightarrow (E_c,  \delta_c, \varphi_c)$ is defined as 
$$ \Psi' \dcomp \Psi := \zeta ^{1/2} ( E_a, E_b, E_c) \Psi' \circ \Psi
$$
where the product in the right-hand side is the usual composition of
maps. 

\begin{prop} \label{prop:cat_demi_forme}
The product $\dcomp$ is well-defined and associative. 
For any object $(E, j , \varphi)$, the identity of $\delta$ is
a unit of $( E, \delta , \varphi)$. So $\demi (S)$ is a category. 
Furthermore each morphism is invertible, its inverse being
its adjoint. 
\end{prop}

\begin{proof}
Because of equation (\ref{eq:def_zeta}), $\Psi' \dcomp \Psi$ is a
morphism. We deduce from the associativity
$$ \Psi_{E_c, E_d} \circ (\Psi_{E_b, E_c} \circ \Psi_{E_a, E_b}) =
(\Psi_{E_c, E_d} \circ \Psi_{E_b, E_c}) \circ \Psi_{E_a, E_b}$$
the cocycle identity 
$$ \zeta ( E_b, E_c, E_d ) \zeta ( E_a, E_b, E_d ) = \zeta ( E_a, E_c,
E_d) \zeta ( E_a, E_b, E_c ) .$$ 
The square root satisfying the same identity, the product $\dcomp$ is
associative. Since $\Psi_{E,E}$ is the identity map, the identity of
$\delta$ is a half-form morphism. Furthermore, 
$$ \Psi_{E_b, E_b } \circ \Psi_{E_a, E_b} = \Psi_{E_a, E_b} , \qquad
\Psi_{E_a, E_b } \circ \Psi_{E_a, E_a} = \Psi_{E_a, E_b}$$
imply that 
$$ \zeta( E_a, E_b, E_b ) = \zeta ( E_a, E_a, E_b) = 1 .$$
So the square root of $\zeta$ satisfies the same equation,
consequently the identity of $\delta$ is
a unit of $(E,  \delta, \varphi)$. 

Since $\overline{E}_{a}$ and $E_{b}$
are Lagrangian spaces, one has
$$ \om ( x - \pi_{E_a, E_b} x , \overline{ \pi_{ E_b, E_a} y} ) = 0 , \qquad  \om(
\pi_{E_a, E_b} x , \overline{y} - \overline{\pi_{E_b, E_a} y} ) = 0$$
for any $x \in E_{a}$ and $y \in E_{b}$. Consequently
$$ \om ( x , \overline{ \pi_{ E_b, E_a} y} ) = \om ( \pi_{E_a, E_b} x
,\overline{y}  )$$
which proves that $\pi_{E_b, E_a}$ is the adjoint of $\pi_{E_a, E_b}$. So
$\Psi_{E_b, E_a}$ is the adjoint of $\Psi_{E_a, E_b}$. Hence the
adjoint $\Psi_{b,a}$ of a half-form morphism $\Psi_{a,b}$ is a
half-form morphism. Furthermore, since the only automorphisms of a
half-form bundle are $\id$ and $-\id$, one has $$ \Psi_{b,a} \dcomp \Psi_{a,b}
= \pm \id, \qquad \Psi_{a,b} \dcomp \Psi_{b,a} = \pm \id.$$
Deforming from $E_a = E_b$ to remove the sign ambiguity, we obtain
that $\Psi_{b,a}$ is the inverse of $\Psi_{a,b}$. 
\end{proof}

\subsection{Symplectic linear category with half-forms} \label{sec:sympl-line-categ}

We introduce now a category $\demi$ with objects the  quadruples
$(S, E, \delta, \varphi)$ consisting of a symplectic vector space with
a positive polarization and a half-form line. Let us define the
morphisms. 

Consider two symplectic vector spaces $S_a$
and $S_b$ with positive polarizations $E_a$ and $E_b$ respectively.   
A symplectic linear isomorphism $g$ from
$S_a$ to $S_b$ sends isomorphically $E_{a}$ to the positive
polarization $g E_{a}$ of $S_b$. So composing the pull-back
by $g^{-1}$ with the morphism $\Psi_{ g E_a, E_b}$ from $\wedge^{\top} (gE_a)^*$ to
$\wedge ^{\top} E_{b}^*$, we obtain an isomorphism 
$$ \Psi_{g, E_a, E_b} := \Psi_{ g E_a, E_b} \circ \bigl( g^{-1} \bigr)^* : \wedge^{\top} E^*_{a}
\rightarrow \wedge^{\top} E^*_{b}.$$
The morphisms
 from $a= (S_a, E_a, \delta_a, \varphi_a)$ to $b= (S_b, E_b, \delta_b,
\varphi_b)$ are defined as the pairs $( g, \Psi)$ consisting of a
symplectic linear isomorphism from $S_a$ to $S_b$ and a linear morphism $
\delta_a \rightarrow \delta_b$ such that  $$ \varphi_b \circ
\Psi^{\otimes 2} = \Psi_{g, E_a, E_b}   \circ
\varphi_a.$$
The composition of a morphism $(g,\Psi): a
\rightarrow b$ with a morphism $(g',\Psi
') : b \rightarrow c$ is defined as
$$ (g', \Psi') \dcomp (g, \Psi) := \bigl( g'g, \zeta ^{1/2} ( g'gE_a,
g' E_b, E_c) \Psi' \circ \Psi \bigr)
$$
where the product in the right-hand side is the usual composition of
maps.  

\begin{prop} \label{prop:cat}
$\demi$ is a category where each morphism is
invertible. 
\end{prop}

\begin{proof} This follows from  proposition
  \ref{prop:cat_demi_forme}. Indeed, given a morphism $(g, \Psi): a
  \rightarrow b$, one
  may identify $S_a$ and $S_b$ with $S$ through symplectomorphisms in
  such a way that $g$ becomes the identity. Then $\Psi$ is a morphism from $(E_a, \delta_a,
  \varphi_a)$ to $( E_b, \delta_b, \varphi_b)$ in the category $\demi
  (S)$. With these identifications, the composition of morphisms in $\demi$
  corresponds to the composition in $\demi (S)$.
\end{proof}

In appendix \ref{sec:linear-quantization}, we will define a functor
from the category $\demi$ to the category of Hilbert spaces.

\section{Metaplectic group} \label{sec:metaplectic-group}

\subsection{The automorphism group}

Consider a fixed symplectic vector space $S$ together with a positive
polarization $E$. Denote by $\Sp (S)$ the group of linear
symplectomorphism. For any $g \in \Sp (S)$, observe that the
endomorphism $\Psi_{ g, E, E}$ of $\wedge ^{\top} E^*$ is the multiplication by $ \det ( g^{-1}\pi_{E, gE} : E
\rightarrow E)$. 

Let $(\delta, \varphi)$ be a half-form line of $(S,E)$. Then
identifying the automorphisms of $\delta$ with complex numbers, the automorphisms
of $(S, E, \delta , \varphi)$ are the pairs $(g, z)$
consisting of a linear symplectomorphism $g \in \Sp (S)$ with a complex number
$z$ such that 
$$ z^2 = \det ( g^{-1}  \pi_{E, gE} : E \rightarrow E) .$$ 
The composition of two automorphisms is given by 
$$ (g', z' ) \circ_{\demi} ( g ,z ) =  \bigl( g' g , \zeta ^{1/2} ( g'gE,
g' E, E) z' z  \bigr)   .$$
Since the previous formulas don't depend on $( \delta,\varphi)$, we
denote by $\Mp (S,E)$ the group of automorphisms of $(S, E, \delta , \varphi)$.
The projection
$$ \Mp (S, E) \rightarrow \Sp (S) $$
is two to one. 
Recall that $\Sp ( S)$ is connected with fundamental
group $\Z$. The metaplectic group is defined as the 2-cover of the symplectic
group. 

\begin{prop} The group $\Mp (S,E ) $ is connected, so it is isomorphic to the
  metaplectic group.  
\end{prop}

\begin{proof}  Choose a line $D$ of $E$. Then for any $\theta \in \R$, consider the
symplectomorphism $g_\theta$ which preserves $D$ and $E$ and such that
for any $x \in E$, 
$$ g_{\theta}  x = \begin{cases} \exp ( i
\theta) x \qquad  \text{ if } x \in D \\  x   \qquad \text{ if } x \perp D
\end{cases} $$
Then $ \theta \rightarrow (g_\theta, \exp ( -i \theta /2 ))$ is a path
of automorphisms connecting
$( \id, 1)$ with $( \id,  -1)$.  
\end{proof}

\subsection{Definition of the index}  \label{sec:definition-index}

Consider the complex structure $j$ of $S$ such that $\ker (j - i \id) = E$. Then $(X, Y) \rightarrow \om (X, jY)$ is a scalar product of
$S$. Let $\sym (S, j)$ be the space of linear endomorphisms of $S$ symmetric
with respect to this scalar product. For any $A \in \sym (S, j)$,
define the square root
$$ \operatorname{det}^{1/2} \bigl(\tfrac{1}{2} \id  +  i A \bigr) $$
in such a way that it depends continuously on $A$ and is positive for
$A=0$. It is well-defined because $\sym (S, j )$ is contractible, as a vector
space. Observe also that any symmetric endomorphism $A$ is diagonalisable with a real
spectrum so that $\tfrac{1}{2} \id + i A$ is invertible.  

Let $\Sp_* (S)$ be the set of symplectic linear isomorphism $g$ of $S$ such that $\id - g $
is invertible. For any such $g$,  it is easily checked that 
$$ A( g) :=  \tfrac{1}{2}  ( \id + g  ) \circ ( \id - g ) ^{-1} \circ j 
$$
belongs to $\sym (S, j)$. Denote by $\Mp_* (S,E)$ the subset of $\Mp (S,E) $
consisting of the pairs $( g, z)$ such that $g \in \Sp _* (S) $. 

\begin{prop} \label{prop:def-index}
For any automorphism $(g, z) \in \Mp_* (S,E) $, one has
$$  z \operatorname{det}^{1/2} \bigl( \tfrac{1}{2} \id +  i A ( g)  \bigr) =  \frac{ i ^{m ( g, z)}
}{|\det(\id - g)  |^{1/2}}$$
where $m(g, z ) \in \Z /  4 \Z $. 
\end{prop}

\begin{proof} Since the inverse of $(\id - g )^{-1} j $ is $-j ( \id -
  g)$, one has
\begin{xalignat*}{2}  
 \tfrac{1}{2} \id +  i A (g) =  &  \tfrac{1}{2} \bigl( -j ( \id  - g )  +
 i (\id + g ) \bigr) ( \id - g )^{-1} j \\ 
=  & -j  ( \overline{\pi}  - \pi g )(\id
- g )^{-1} j  
\end{xalignat*}
where $\pi = \tfrac{1}{2}( \id -  i j )$ is the projector of $S
\otimes \C$ with image $E$ and kernel $\overline{E}$. From this one
deduces that
\begin{gather*}  
 \det \bigl( \tfrac{1}{2} \id +  i A (g) \bigr) =   \frac{\det(
   \overline{\pi} - \pi g) }{\det ( \id - g )} 
=  ( -1)^n \frac{\det ( \pi g : E \rightarrow
E )}{ \det ( \id - g ) }   
\end{gather*} 
The inverse of $\pi_{E, gE}$ is the restriction of $\pi$ from $gE$ to $E$. So 
\begin{gather} \label{eq:2}
 z^2 = \det ( g^{-1} \pi_{E, gE} : E \rightarrow E) = \operatorname{det}^{-1} ( \pi g : E \rightarrow
E )  .
\end{gather}
The two previous equations imply that  
\begin{gather} \label{eq:square}
 z^ 2 \det \bigl( \tfrac{1}{2} \id +  i A (g) \bigr) =  \frac{(-1)^n}{  \det ( \id - g )
},
\end{gather}
which concludes the proof.
\end{proof} 

We call $m(g,z)$ the index of $(g,z)$. This defines a locally constant
map $m$ from $\Mp _* (S,E)$ to $ \Z / 4 \Z$ which distinguishes the
components of $\Mp_* (S,E)$, as proves the following proposition.
\begin{prop} $\Mp_* (S,E) $ has four components: $m^{-1} (0)$, $m^{-1} (1)$, $m^{-1} (2)$ and $m^{-1} (3)$. 
\end{prop}
 Observe also that $m( -\id , e^{i  \frac{\pi}{2} p } )= p $ for $p
 =n$ mod 2. When $S$ is two-dimensional one can can easily
 characterize the index of $(g,z)$ in term of the trace of $g$ and the
 argument of $z$, cf. equation (\ref{eq:ind2d}). 

\begin{proof}
It is proved in \cite{CoZe} lemma 1.7,   that $\Sp_* (S)$ has two
components, distinguished by the sign of $\det ( g - \id)$. So $\Mp
_*(S,E)$ has at most 4 components. To conclude it suffices to prove that
$m$ is onto. This follows from equation (\ref{eq:square}), because the
right hand side can be positive and negative. Actually, it is also
proved in \cite{CoZe} lemma 1.7 that any loop in $\Sp _* (S)$ is
contractible in $\Sp (S)$, which gives another proof that $\Mp _* (S,E)$ has 4 components. \end{proof}

\subsection{A characterization of the index} \label{sec:char-index}

Assume $S$ is two-dimensional.
By choosing a basis $(e,f)$ of $S$
such that the symplectic product of $e$ with $f$ is equal to 1, we
identify $S$ with $\R^2$ and $\Sp (S)$ with $\Sl (
2, \R)$.  For any $g \in \Sp (S)$, 
$$ \det ( \id - g ) = 2 - \operatorname{tr}( g ) .$$
So the set of $g$ with $\det ( \id - g ) <0$ consists of the
hyperbolic elements with negative trace, whereas the set of $g$ with $ \det (
\id - g ) >0$ consists of the elliptic elements together with the
hyperbolic ones with a positive trace. 

Recall that $\Sl ( 2, \R)$ is diffeomorphic to the product of the circle
and the unit disc.  Such a diffeomorphism is the map sending $\te \in S^1$ and
$(u,v) \in D $ into
$$ g ( \te , u , v ) :=  \bigl( 1 - u^2 - v^2 \bigr) ^{-1/2}
\left( \begin{matrix} \cos ( \te ) + u   & - \sin ( \te) + v \\ \sin (
    \te ) + v & \cos ( \te ) - u  \end{matrix} \right) $$
where $D = \{(u,v) / \;   u^2 + v^2
< 1 \}$.

Let $E$ be the positive polarization generated by $e - i f$ and let
$\Mp(S,E)$ be the associated metaplectic group. Let us parametrize
$\Mp (S,E)$ by $S^1 \times D$ 
$$ ( \te, u,v) \rightarrow  ( g ( - 2 \te , u, v) , e^{i \te } ( 1
- u ^2 - v ^2 )^{-\frac{1}{4}} )  $$
Then $\Mp _* (S,E)$ is the image of  $\bigl\{ \cos ( 2 \te) \neq \bigl( 1 -
u ^2 - v^2 \bigr)^{1/2} \bigr\}$. We find easily the index of any
element of $ \Sp_* (S)$ by computing explicitly the index
of one element in each component. cf. figure \ref{fig:index}.

\begin{figure}
\begin{center}
\input{index.pstex_t}  
\caption{the index} \label{fig:index}
\end{center}
\end{figure}
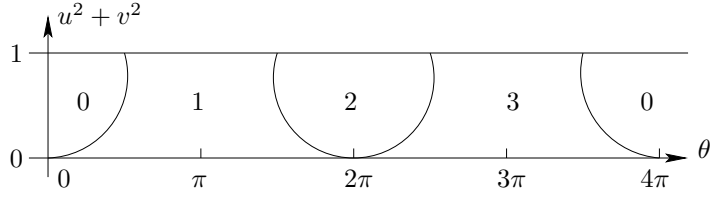

Observe that when the argument of $z$ runs over an interval $\bigl[ k
\tfrac{\pi}{2}, (k+1) \tfrac{\pi}{2} \bigr[$ with $k \in \Z$, the index $m(g,z)$ takes two
distinct values depending on the sign of $\operatorname{tr}(g) -
2$. Checking the various cases, one obtains the following formula
\begin{gather} \label{eq:ind2d} 
 m( g, z) = k + \tfrac{1}{2} ( 1 - ( -1) ^{k + \epsilon})
\end{gather}
where $k$ and $\epsilon$ are determined by 
$$ \arg (z) \in \bigl[  k
\tfrac{\pi}{2}, (k+1) \tfrac{\pi}{2} \bigr[ , \qquad \epsilon =
\begin{cases} 0 \text{ if } \operatorname{tr} ( g) > 2 \\ 1 \text{
    otherwise}.
\end{cases} $$ 

The formula does not depend on the parametrization. Unfortunately such a simple description doesn't generalize
in higher dimension. Nevertheless we can characterize the index in any
dimension by
considering product. 

Let $S_1$ and $S_2$ be two symplectic vector spaces with
  positive polarizations $E_1$ and $E_2$. Then $E= E_1 \times E_2$ is
  a positive polarization of $S = S_1 \times S_2$. We have a
  morphism from $\Mp (S_1,E_1) \times \Mp ( S_2, E_2) $ to $\Mp ( S,E)$
  sending $((g_1, z_1),$ $(g_2, z_2))$ into $( g_1 \times g_2 , z_1
  z_2)$. Furthermore, 
\begin{gather} \label{eq:indprod}
 m (  g_1 \times g_2 , z_1
  z_2) = m ( g_1, z_1) + m ( g_2 , z_2). 
\end{gather}
So the image of $\Mp_* (S_1,E_1) \times \Mp_* ( S_2,E_2)$ meets each
connected component of $\Mp_* (S,E)$. This gives the following
characterization. 
\begin{prop}  The collection $(m_{S,
  E}: \Mp (S, E) \rightarrow \Z/4 \Z)$, where $(S,E)$ runs over the symplectic vector space endowed
with a positive polarization, is the unique collection of continuous
map satisfying  (\ref{eq:ind2d}) for any two dimensional space and
 (\ref{eq:indprod}) for any product. 
\end{prop}  

\subsection{Unitary group of $E$}  

The subgroup of $\Sp (S)$ consisting of the elements commuting with
$j$ is isomorphic with the unitary group of $E$. The isomorphism is
the map $\iota $ sending $h \in U(E)$ to $g \in \Sp (S)$ whose
complexification acts as
$$ g ( x ) = \begin{cases} h (x) \text{ if } x \in E \\ \overline{h
   } (x) \text{ if } x \in \overline E
   \end{cases}
   $$
Denote by $U_*(E)$ the subset of $U(E)$ consisting of the $h$ such
that $h -id $ is invertible. Obviously $h \in U_*(E)$ iff $\iota (h) \in
\Sp_* (S)$. Next lemma will be used to compare our trace estimates
with the holomorphic Lefschetz fixed point formula.

\begin{lem} \label{lem:indexholom}
For any $h \in U_*(E)$, we have 
$$ \operatorname{det}^{1/2} \bigl( \tfrac{1}{2} \id +  i A ( \iota ( h)
)  \bigr) =  \frac{1}{
  \det ( \id - h^{-1})} .
$$
\end{lem}

\begin{proof} 
Let $(e_i)$ be an orthonormal basis of $E$ diagonalising $h$. Then the
matrix of $A(g)$ in the base $(e_1, \ldots , e_n, \overline{e}_1,
\ldots , \overline{e}_n)$ is given by 
$$\left( \begin{matrix} D   & 0  \\ 0 &  \overline{D}  \end{matrix} \right)$$
with $D$ the diagonal matrix with entries $ d_i = \frac{i}{2} (1+u_i)
( 1 - u_i)^{-1} 
$. Here the $u_i$'s are the eigenvalues of $h$, $h(e_i) = u _i e_i$. A
straightforward computation using that $u_i$ is a 
complex number with modulus 1 gives 
$$ ( \tfrac{1}{2} + i d_i ) ( \tfrac{1}{2} +i \overline{d_i}) = \frac{1}{ (1- \overline{u}_i)^2}$$
so that
$$ \operatorname{det} \bigl( \tfrac{1}{2} \id +  i A ( j(h))  \bigr) =  \frac{1}{
  \operatorname{det}^2 ( \id - h^{-1})  }$$
This proves the result up to a plus or minus sign. $U_*(E)$ being
connected, it suffices now to check the result for one
element $h$. It is obvious for $h = -\id$, because $A(\iota (h) ) =0$. 
\end{proof}

Let $\tilde{U} (E)$ be the subgroup of $U(E) \times \C$
$$ \tilde{U} (E) = \{ ( h, z) / \; z^2 = \det h \}$$
It is isomorphic to the twofold cover of $U(E)$. By equation (\ref{eq:2}), $\iota$ lifts
to the embedding from $\tilde{U}(E)$ to $\Mp (S,E)$ sending $(h, z)$ to
$(\iota (h) , z^{-1})$. This map is a group morphism.  Finally observe that
$\tilde{U}_*(E)$ has 2 components, one containing $( -\id , \epsilon)$ and the
other $( -\id , -\epsilon)$ with $\epsilon ^2 = ( -1) ^n$. So the
index $m$ takes two distinct values on $\tilde{U}_*(E)$.

\subsection{Metalinear group of a Lagrangian subspace} 

Let $\Lambda$ be a Lagrangian subspace of $S$. Let $\Sp ( S, \Lambda,
j \Lambda)$ be the subgroup of $\Sp (S)$ consisting of the elements
preserving $\Lambda$ and $j \Lambda$. Let $\Ml ( \Lambda)$ be the
metalinear group  of $\Lambda$, i.e. the subgroup of
$\Gl ( \Lambda) \times \C^*$ consisting of the pairs  $(h , z)$ such
that $ z^2 = \det h $. 

\begin{prop}
For any $(h,z) \in \Ml ( \Lambda)$, there is a unique pair $( g, z') \in \Sp ( S, \Lambda,
j \Lambda) \times \C^*$ such that
\begin{itemize} 
\item $h$ is the restriction of $g$ to $\Lambda$. 
\item $z'/z \in \R_+$.  
\end{itemize}
The map $ j_\Lambda : \Ml ( \Lambda) \rightarrow \Mp
(S,E)$ sending $(h,z )$ into $(g, z')$ is an injective group morphisms,
with image the set of $(g,z')$ such that $g \in  \Sp ( S, \Lambda,
j \Lambda)$. 
\end{prop}

\begin{proof} 
Let $(e_i)$ be a basis of $\Lambda$, orthonormal for the
scalar product $\om ( x, j y)$. Then $g \in \Sp ( S, \Lambda,
j \Lambda)$ if and only if its matrix in the basis $e_1, \ldots , e_n,
je_1$, $\ldots$, $j e_n $ is of the form 
\begin{gather} \label{eq:3}
\left( \begin{matrix} B   & 0  \\ 0 &  C  \end{matrix} \right)
\end{gather}
with $B C^t = \id$. So $g$ is determined by its restriction $h$ to
$\Lambda$ which is an arbitrary element of $\Gl ( \Lambda )$. 

Denote by $\pi$  the projection onto $E$ with kernel $\overline{E}$. Then
the matrix of $ \pi g : E \rightarrow E$ in the basis $(e_i - i je_i)_i$
is $ \frac{1}{2} (  B + C)$. 
Furthermore $\det ( B + C)$ and $\det B$ have
the same sign.  So $\det ( \pi g : E \rightarrow E)$ is
real and have the same sign of $\det (h)$. This proves the existence and unicity of $z'$.   

To prove that $j_\Lambda$ is a morphism, we have to show
that $ \zeta^{1/2} ( g'g E, g'E, E)$ 
is a positive number for any $g, g' \in \Sp ( S, \Lambda,
j \Lambda)$. By the first part of the proof, $\zeta ( g'g E,
g'E, E)$  is real. So it suffices to prove it is positive for one pair $(h, h')$ in each component
of $\Gl(\Lambda ) \times \Gl ( \Lambda)$. If any two of three
polarizations $E$, $F$ and $G$ are equal, then 
$$ \zeta^{1/2} (G, F,  E) = 1 . $$
One deduces  that $\zeta^{1/2} ( g'g E, g'E, E) =1 $ for $(h, h') = ( \id, \id )$, $( \id, k )$, $( k , \id)$ and
$(k,k)$ where $k $ is any involution of $\Lambda$ with negative
determinant.  
\end{proof}

\begin{prop} 
For any $(h,z ) \in \Ml ( \Lambda )$ such that $h - \id $ is
invertible, the index $m$ of $j_\Lambda ( h , z)$ is determined by 
$ z =  i ^m |z| $. 
\end{prop}

\begin{proof} 
One has to prove that $ \operatorname{det}^{1/2} \bigl( \tfrac{1}{2}
\id + i A(g) \bigr)$ is
positive for any $g \in \Sp ( S , \Lambda, j \Lambda)$. If the matrix
of $g$ is (\ref{eq:3}), the one of $A(g)$ is 
$$ -\tfrac{1}{2} \left( \begin{matrix} 0   & D \\ D^t &  0 \end{matrix} \right)
\qquad \text{ with }  D =( 1 +B) ( 1- B)^{-1} $$ 
So $$\det \bigl( \tfrac{1}{2} \id + i A(g) \bigr) = \det \bigl( \tfrac{1}{2} (
\id + D^t D) \bigr).$$ Deforming $A(g)$ to $0$ through a radial
homothety we obtain that the square root of $\operatorname{det}
\bigl( \frac{1}{2} \id +  i A (g) \bigr )$ is positive.  
\end{proof}

\section{Generalized half-form lines} \label{sec:gener-half-form}

Let $p$ be a positive integer. A generalized half-form line of a symplectic vector space $S$ equipped with a positive
polarization $E$ is a complex line $\delta$ together with an
isomorphism 
$$\varphi : \delta^{\otimes 2 p} \rightarrow  \bigl( \wedge
^{\operatorname{top}} E^* \bigr)^{\otimes p}$$
We have a category whose objects are the quadruples
$(S,E, \delta, \varphi)$. The morphisms from $(S_a, E_a, \delta_a,
\varphi_a)$ to $(S_b , E_b , \delta_b, \varphi_b)$ are the pairs
consisting of a linear symplectomorphism $g : S_a \rightarrow S_b$ together
with a morphism $\Psi : \delta_a \rightarrow \delta_b$ such that 
 $$ \varphi_b \circ
\Psi^{\otimes 2 p } = \Psi_{g, E_a, E_b}^{\otimes p}    \circ
\varphi_a $$
where the map $\Psi_{g, E_a, E_b}$ is defined as in section
\ref{sec:sympl-line-categ}.   
The composition of a morphism $(g,\Psi): a
\rightarrow b$ with a morphism $(g',\Psi
') : b \rightarrow c$ is defined as
$$ (g', \Psi') \dcomp (g, \Psi) := \bigl( g'g, \zeta ^{1/2} ( g'gE_a,
g' E_b, E_c) \Psi' \circ \Psi \bigr)
$$
where the product in the right-hand side is the usual composition of
maps.  

The automorphism group $\Mp _p (S,E)$ of $(S,E, \delta, \varphi)$
consists of the pair $(g, z)$ where $g$ is a linear symplectomorphism
of $S$ and $z$ a complex number such that
$$  z^{2p}  =  \operatorname{det} ^p  ( g^{-1}  \pi_{E, gE} : E
\rightarrow E) .$$
Let $U_{2p}$ be the group of $2p$-th roots of unity. The map sending
$((g,z), u ) $ to $(g, zu)$ is a surjective morphism
$$ \Mp ( S,E) \times U_{2p} \rightarrow \Mp_p (S,E)$$
with kernel $\bigl\{ ( (\id, 1),1), (( \id , -1) , -1 )
\bigr\}$. For any element $(g,z) \in \Mp_p (S,E)$ such that $1$ is not
an eigenvalue of $g$, we define its index in the following
way. Set $p' = p$ if $p$ is even and $p' = 2p $ if $p $ is odd. Then
$m_p(g,z)$ is the unique element of $\Z$ mod $2p' \Z$ such that
\begin{gather*}
  z \operatorname{det}^{1/2} \bigl( \tfrac{1}{2} \id +  i A ( g)
\bigr) =   \frac{ e^{i \frac{\pi}{p'} m _p( g, z)}
}{|\det(\id - g)  |^{1/2}}
\end{gather*}
The existence of $m_p (g,z)$ follows from proposition \ref{prop:def-index}.

\section{Quantization of K{\"a}hler manifolds} \label{sec:quant-kahl-manif}

\subsection{Hilbert space} 

Consider a compact K{\"a}hler manifold $M$ with a prequantization bundle
$L$, that is $L$ is a holomorphic Hermitian
line bundle such that the curvature of its Chern connection is $\frac{1}{i} \om $ where $\om$ is the fundamental two-form of
$M$. Let $K\rightarrow M$ be a holomorphic Hermitian line bundle.
Define the sequence of vector spaces 
$$\Hilb_k  := \bigl\{  \text{holomorphic sections of } L^k \otimes K \bigr\}, \qquad k =1,2,...$$ 
Since $M$ is compact, $\Hilb_k $ is a finite dimensional vector
space. It has a natural scalar product  defined by means of the
Hermitian structure of $ L^k \otimes K $
and the Liouville measure of $M$.

\subsection{Fourier integral operators} \label{sec:four-integr-oper}

Consider a symplectomorphism $\Phi : M \rightarrow M$ together with an
automorphism $\Phi_L$ of the bundle $L$ lifting
$\Phi$ and preserving the connection and the metric. 
To these data is associated a space $\Fourier ( \Phi_L)$ of Fourier integral
operators, that we define now.

 Consider a family of operators $(S_k: \Hilb _k
 \rightarrow \Hilb _k, \; k =1,2 , \ldots) $.  The scalar product of $\Hilb _k
$ gives us an isomorphism 
$$ \Hom (\Hilb _k , \Hilb _k  ) \simeq  \Hilb _k 
\otimes \overline{\Hilb}
_k  .$$
The latter space can be regarded as the space of holomorphic sections of  
$$ (L^k \otimes K ) \boxtimes (\overline{L}^k \otimes \overline{K})
\rightarrow M^2,$$
where $M^2$ is endowed with the complex structure $(j, -j)$. The
section $S_k (x,y)$ associated in this way to $S_k$ is its Schwartz kernel. 

By definition $(S_k)$
is a Fourier integral operator of $\Fourier ( \Phi_L)$ if 
\begin{gather} \label{def:FIO}
 S_k(x,y) =  \Bigl( \frac{k}{2\pi} \Bigr)^{n} F^k(x,y) g(x,y,k) + O
(k^{-\infty}) \end{gather}
where 
\begin{itemize}
\item 
$F$ is a section of  $L \boxtimes \bar{L}
\rightarrow M^2$ such that  $\| F(x,y)
\| <1 $ if $x \neq \Phi(y) $, 
$$ F (\Phi ( x) ,x) = \Phi_L(u)  \otimes \bar{u}, \quad \forall u \in L_x \text{ such that }
  \| u \| = 1, $$
and $ \bar{\partial} F \equiv 0 $
modulo a section vanishing to any order along the graph of $\Phi ^{-1}$.
\item
  $g(.,k)$ is a sequence of sections of  $ K  \boxtimes
  \bar{K}  \rightarrow M^2$ which admits an asymptotic expansion in the $\Ci$
  topology of the form 
$$ g(.,k) = g_0 + k^{-1} g_1 + k^{-2} g_2 + ...$$
whose coefficients satisfy $\bar{\partial} g_i \equiv 0  $
modulo a section vanishing to any order along the graph of $\Phi ^{-1}$.
\end{itemize}

Let us define the principal symbol of $(S_k)$ to be the map $x \rightarrow
g_0(\Phi ( x) ,x)$. Using the Hermitian structure of $K$, we regard it as a
section of the bundle $\Hom (K , \Phi_*K) \rightarrow M$. The principal symbol map 
$$\sigma :
\Fourier ( \Phi_L ) \rightarrow \Ci(M, \Hom (K , \Phi_* K))$$
satisfies the expected property.

\begin{theo} \label{the:P0}
The following sequence is exact
$$ 0 \rightarrow \Fourier (\Phi_L) \cap O(k^{-1}) \rightarrow 
\Fourier ( \Phi_L ) \xrightarrow{\sigma} \Ci(M,\Hom (K , \Phi_*
K)) \rightarrow 0, $$
where the $O(k^{-1})$ is for the uniform norm of operators. 
\end{theo}

Consider  two symplectomorphisms $\Phi$ and $\Phi ' $. Define the product of
two  sections $\Psi$, $\Psi'$ of $\Hom ( K, \Phi_*K )$ and
$\Hom ( K, \Phi' _*K)$ respectively as the section of
$\Hom ( K , (\Phi' \circ \Phi )_* K) $ given at $x$ by 
\begin{gather} \label{eq:comp_mor}
 \Psi' \circ_\demi \Psi (x) =  \zeta^{1/2} (  g'_{\Phi(x)} g_x 
E_x, g'_{\Phi(x)}  E_{\Phi ( x)},  E_{(\Phi' \circ \Phi)(x)} 
)   \Psi' (
\Phi (x))  
\circ \Psi (x) 
\end{gather}
where $E_y = T^{1,0}_y M$, $g_y = T_y \Phi$ and $g'_y = T_y \Phi'$ for
$y= x$, $\Phi (x)$ or $\Phi' ( \Phi (x))$.  Here the product on the right hand
side is the usual composition of homomorphism.

\begin{theo} \label{the:P1}
Let $\Phi_L$ and $\Phi_L'$ be two automorphisms of the prequantum
bundle $L$ lifting $\Phi$ and $\Phi'$ respectively. If $T \in \Fourier (\Phi_L)$ and $S
\in \Fourier (\Phi_L')$, then $S\circ T$ is a Fourier integral operator of
$\Fourier (\Phi_L ' \circ \Phi_L )$. Its symbol is given by  
$$ \si ( S \circ T ) =  \si (S) \circ_\demi \si (T).
$$
\end{theo}

Theorems \ref{the:P0} and \ref{the:P1} are consequences of theorems 3.1 and 3.2 of
\cite{oim_mc}.

\subsection{Trace estimate}

Consider a symplectomorphism $\Phi$ of $M$ together with a lift $\Phi_L$ to the
prequantum bundle. Assume that the graph of $\Phi$ intersects
transversally the diagonal. 

\begin{theo} \label{the:trace-estimate}
For any $(S_k) \in \Fourier ( \Phi_L)$ with symbol $\Psi$, we have
$$ \operatorname{Tr} ( S_k) =  \sum_{x = \Phi (x) } z_x
\operatorname{ det}^{1/2} \bigl( \tfrac{1}{2}  \id + i A_x \bigr) \; u_x^k $$
where for any fixed point $x$ of $M$, 
\begin{itemize} 
\item $z_x$ and $u_x$ are the traces of
$\Psi(x) : K_x \rightarrow K_x  $ and $ \Phi_L (x) : L_x \rightarrow
L_x$ respectively. 
\item $A_x = \tfrac{1}{2}  (\id + T_x \Phi) \circ ( \id - T_x \Phi)^{-1}
  \circ  j_x $
  and the square root of the determinant is determined as in section
  \ref{sec:definition-index}. 
\end{itemize}
\end{theo}

\begin{proof}
By assumption the Schwartz kernel of $S$ has the form
(\ref{def:FIO}). One has
$$ \operatorname{Tr} ( S_k)  = \Bigl( \frac{k}{2\pi} \Bigr)^{n} \int _M F^k(x,x)
f(x,x,k) \mu _M(x)  + O
(k^{-\infty}) $$ 
Since $|F(x,y)| < 1 $ outside $\Gamma = \{ ( \Phi (x) , x ) / \; x \in M
  \}$, one can restrict the integral to a neighborhood of the fixed
  point set of $\Phi$. Let us write on a
neighborhood of $\Gamma $
$$\nabla ^{L \boxtimes \overline{L}} F = \beta \otimes F $$
In the proposition 2.2 of \cite{oim_qm}, the first derivatives of $\beta$
along $\Gamma$ are computed. Denote by $E_x $ the space $T^{1,0}_x M$. 
\begin{lem} \label{lem:derF}
The form $\beta$ vanishes along $\Gamma$. For any vector fields $X$ and $Y$ of $M^2$, one has
  at any point of $\Gamma$  
$$ \der_X \langle \beta , Y \rangle  = \tfrac{1}{i}
\om_{M \times M^-} ( q (X) , Y) $$
where $\om_{M \times M^-} $ is the symplectic form of the product of
$(M, \om)$ with $(M, -\om)$. And for any $x \in M$,  $q_{(\Phi(x), x)}$ is the projection of $T_{\Phi(x)} M \times T_x M$
onto $\overline{E}_{\Phi(x)} \times E_x$ with kernel the tangent space
of $\Gamma $ at $(\Phi (x), x)$. 
\end{lem}
Let us write on a neighborhood of a fixed point $x_0$ of $\Phi$
$$ F (x,x) = \exp ( - \varphi (x))$$ 
where $\varphi$ is a complex valued function. By lemma
\ref{lem:derF}, the first derivatives of $\varphi$
vanishes at $x_0$. 
\begin{lem} The Hessian of $\varphi$ at $x_0$ is given by  
$$ \operatorname{Hess} \varphi (X, X') (x_0) =  \om ( \bigl(
\tfrac{1}{2} \id + i A_{x_0} \bigr)^{-1}  X, jX' )$$
for any tangent vectors $X, X'$ of $M$ at $x_0$. 
\end{lem}
\begin{proof} 
By lemma \ref{lem:derF}, the Hessian is given by
$$ \operatorname{Hess} \varphi (X, X') (x_0) =  i \om_{M \times M^-} (
q (X,X) , (X',X')  )$$
Let $\pi$ be the projector of $T_x M \otimes \C$ with image $E_x$ and
kernel $\overline{E}_x$. Then $$q (X,X) = (\overline{\pi} (V), \pi (V))$$
for a unique $V \in T_xM \otimes \C$. Using that $\pi = \frac{1}{2} ( \id -
i j_x )$, we obtain that
\begin{xalignat*}{2} 
\operatorname{Hess} \varphi (X, X') (x_0) = & i \om ( \overline{\pi}
(V), X') - i \om ( \pi (V), X')   \\ = & 
  \om ( V, j_xX')
\end{xalignat*}
To compute $V$ one has to solve the following system
$$ \begin{cases} X = T_x  \Phi (Y)  + \frac{1}{2} ( \id +
i j_x ) V \\  X = Y  + \frac{1}{2} ( \id -
i j_x ) V
\end{cases}$$
Adding and subtracting both equations, we obtain 
$$ \begin{cases} 2X = ( \id +  T_x  \Phi)  (Y)  +  V \\  0 =  ( \id -
  T_x  \Phi)  (Y)  - i j_x  V
\end{cases}$$
With the second equation, we compute $Y$ in terms of $V$. Inserting
the result in the first equation, we obtain
$$ V =  \bigl(
\tfrac{1}{2} \id + i A_{x_0} \bigr)^{-1}  X$$
with proves the lemma.
\end{proof}
Now the theorem follows from stationary phase lemma by using that
$\exp ( - \varphi(x_0) ) = u_{x_0}$ and $g( x_0, x_0, k ) = z_{x_0} +
O(k^{-1})$.  Observe furthermore that the Liouville measure is the
Riemannian volume of the metric $\om (X, j_x Y)$. So for an orthogonal
basis $X_1, \ldots, X_{2n}$, we have
$$ \mu_M( X_1 \wedge \ldots \wedge X_{2n} ) =1 $$ 
and 
$$ \det \bigl( \operatorname{Hess} \varphi ( X_i, X_j) (x_0)
\bigr)_{i,j}  =
\operatorname{det} ^{-1} \bigl( \tfrac{1}{2} \id + i A_{x_0} \bigr)
$$ which leads to the result. 
 \end{proof}

\subsection{The holomorphic Lefschetz fixed point formula} 

As in the previous section, consider a symplectomorphism $\Phi$ of $M$ together with a lift $\Phi_L$ to the
prequantum bundle. Assume that $\Phi$ is a holomorphic map. Since
the holomorphic structure of $L$ is characterized by the connection,
$\Phi_L$ is a holomorphic bundle homomorphism. Consider a holomorphic
bundle homomorphism $\Psi$ of $K$ lifting $\Phi$. Then we have a map
$$ (\Phi_L^k \otimes \Psi)_* : \Hilb_k \rightarrow \Hilb_k$$
defined as the inverse of the pull-back by $\Phi_L^k \otimes \Psi$.
More generally, $\Phi_L^k \otimes \Psi$ acts on the $q$-th
cohomology group of the sheaf of holomorphic sections of $L^k \otimes
K$. One defines the holomorphic Lefschetz number
$$ L ( \Phi_L^k \otimes \Psi) := \sum_{ q=0}^n ( - 1)^q
\operatorname{Tr} \bigl( (\Phi_L^k \otimes \Psi) _* |_{ H^q (M, L^k
  \otimes K)} \bigr) $$  
Then assuming that the graph of $\Phi$ intersects transversally the
diagonal, the holomorphic Lefschetz fixed point theorem,
cf. \cite{AtBo} theorem 4.12, says
$$ L ( \Phi_L^k \otimes \Psi) = \sum_{x = \Phi(x) } \frac{z_x u_x^k }
{ \det \bigl( \id - h_x^{-1}  \bigr)}$$ 
where the complex numbers $z_x$ and $u_x$ are defined as in theorem \ref{the:trace-estimate}
and $h_x $ is the holomorphic tangent map of $\Phi$ at $x$, that is the restriction of $ T_x \Phi \otimes \C$ to $E_x = T^{1,0}_xM$. 

When $k$ is sufficiently large, Kodaira's vanishing theorem implies
that $H^q (M, L^k \otimes K) =0$ for every positive $q$. If the latter
is the case, the holomorphic Lefschetz number is the trace of the action of
$ \Phi_L^k \otimes \Psi$ on $\Hilb_k$. 

Furthermore, the family of operators $( (\Phi_L^k \otimes \Psi) _*
|_{ \Hilb_k }, k=1,2,\ldots ) $ is a Fourier integral operator of $\Fourier( \Phi_L)$
with symbol $\Psi$. This is an easy consequence of the fact that the
sequence $(\id_{\Hilb_k}, k=1,2, \ldots)$  belongs to $\Fourier ( \id_L)$ and
has the symbol $\id_K$. 

So theorem \ref{the:trace-estimate} gives the asymptotic behaviour of the
Lefschetz numbers. That the result agrees with holomorphic Lefschetz
fixed point theorem is a consequence of lemma \ref{lem:indexholom}.

\section{Quantization with half-form bundle} \label{sec:quant-half-form}

\subsection{Hilbert spaces} 
Let $p$ be a positive integer. Let $(\delta,
\varphi)$ be a generalized half-form bundle of $M$, i.e. $\delta$ is a complex line bundle over $M$
and $\varphi$ is an isomorphism form $\delta^{\otimes 2p }$ to
$\bigl( \wedge^{n,0} T ^*M \bigr)^{\otimes p}$. So at each point $x \in M$, we have a positive
polarization $$T^{1,0} _{x} M = \ker( j_x - i \id)$$  of $T_x M$ and
a generalized half-form line $( \delta_x, \varphi_x)$ of this
polarization. 

The half-form bundle $\delta$ has a natural metric and holomorphic structure
such that $\varphi$ is an isomorphism of Hermitian
holomorphic bundle. We apply the previous constructions with $K
= \delta$, which defines the  Hilbert space $\Hilb _k$. 

\subsection{Unitary maps} \label{sec:unitary-maps}

Let $\Psi$ be an
automorphism of the bundle $\delta$ lifting a symplectomorphism $\Phi$
of $M$. One says that
$\Psi$ is a half-form bundle automorphism  if for any point $x$
of $M$, $$\Psi(x) : \delta_x \rightarrow \delta_{\Phi(x)}$$ is a
morphism of half-form lines, cf. section \ref{sec:sympl-line-categ}
for $p=1$ and section \ref{sec:gener-half-form} for any $p$. Observe that for any two
half-form bundle automorphisms $\Psi$ and $\Psi'$, the product $\Psi'
\circ_\demi \Psi$ defined in (\ref{eq:comp_mor}) is a half-form bundle morphism. 

\begin{theo} \label{sec:unitary-maps-comp}
For any automorphisms $\Phi_L, \Psi$ of the prequantum
bundle $L$ and the half-form bundle $\delta$  respectively which lift the same
symplectomorphism of $M$, let $U( \Phi_L, \Psi)$ be the set of unitary
Fourier integral operators of $\Fourier( \Phi_L)$ with symbol
$\Psi$. Then 
\begin{itemize}  
\item $U( \Phi_L, \Psi)$ is not empty.
\item $ T_k \in U(\Phi_L, \Psi) \text { and } T'_k \in U(\Phi_L', \Psi')
\Rightarrow (T'_k T_k
) \in  U(\Phi_L' \circ \Phi_L , \Psi' \circ_\demi \Psi )
$
\item 
$ U( \id_L, \id_\delta)$ consists of the sequences $\exp ( i k^{-1}
T_k)$ where $(T_k)$ runs over the self-adjoint operators of
$\Fourier ( \id _L)$.   
\end{itemize}
\end{theo}

\begin{proof} We only give an outline since the ideas of the proof are
  standard. To show that $U( \Phi_L, \Psi)$ is not empty, consider a
  Fourier integral operator $(T_k)$ of $\Fourier( \Phi_L)$ with symbol
$\Psi$. Then its adjoint is a Fourier integral operator of $\Fourier(
\Phi_L^{-1})$ with symbol
$\Psi^*$. By proposition \ref{prop:cat_demi_forme}, $\Psi^*$ is the
inverse of $\Psi$. So by theorem \ref{the:P1}, $(T_k ^* T_k)$ is a Fourier
integral operator of $\Fourier(\id_L)$ with symbol the
identity. $\Fourier(\id_L)$ is the algebra of Toeplitz operators. By
ellipticity, $T_k ^* T_k$ is an invertible self-adjoint operator when $k$ is
sufficiently large. By changing the first values of $T_k$, $T_k ^*
T_k$ is invertible for any $k$. Then using the functional calculus of Toeplitz
operators (cf. proposition 12 of \cite{oim_bt}), one proves that $ (T_k^* T_k)^{-1/2}$ is a Toeplitz operator
with principal symbol equal to $1$. This implies that $\bigl( T_k (T_k^*
T_k)^{-1/2} \bigr)$ belongs to $U( \Phi_L, \Psi)$. 

The second part of the theorem follows directly from theorem \ref{the:P1}. To
show the last part, one constructs the operator $T_k$ by successive
approximations using the functional calculus of
Toeplitz operators. 
\end{proof}

\subsection{Trace estimates}

Consider two automorphisms $\Phi_L, \Psi$ of the prequantum
bundle $L$ and the half-form bundle $\delta$  respectively which lift the same
symplectomorphism $\Phi$ of $M$. 
\begin{theo}  \label{sec:trace-estimates}
Assume that the graph of $\Phi$
intersects transversally the diagonal of $M^2$. Then for any $(T_k)
\in U( \Phi_L, \Psi)$, one has 
$$ \operatorname{ Tr} ( T_k) = \sum_{x = \Phi (x) } \frac{ e^{i
    \frac{\pi}{p'} m _p( g_x, z_x)} u_x^k } { | \det ( \id - g_x ) |^{1/2}} + O(k^{-1}) $$
where for any fixed point $x$ of $\Phi$, 
\begin{itemize} 
\item $g_x$ is the linear tangent map to $\Phi$ at $x$ and $z_x \in \C$ is the trace of the endomorphism $\Phi_{\delta,x} :
  \delta_x \rightarrow \delta_x$
\item $p' = p$ (resp. $2p$) if $p$ is even (resp. odd) and  $m_p ( g_x, z_x)
  \in \Z$ mod $2p'\Z$ is the index
  defined in section~\ref{sec:gener-half-form}. 
\item  $u_x \in \C$ is the trace of the endomorphism $\Phi_{L,x} :
  L_x \rightarrow L_x$
\end{itemize}
\end{theo}

This is an immediate consequence of theorem \ref{the:trace-estimate}
and the definition of the index.

\appendix

\section{Linear Quantization} \label{sec:linear-quantization}

In this appendix we define a functor from the category of symplectic
space with polarization and half-form to the category of Hilbert
space. 

\subsection{Hilbert space} 
Let $S$ be a symplectic vector space. Consider the trivial bundle
$L_S$ with base $S$, fiber $\C$ and endowed with the connection $d + \frac{1}{i} \al$
where $\al \in \Om^1 ( S)$ is given by 
$$ \al |_x ( y) = \tfrac{1}{2} \om (x, y) .$$
Let $E$ be a positive polarization and $(\delta, \varphi)$ be a
half-form line. Abusing notation, we denote by $\delta$ the trivial bundle with
base $S$, fiber $\delta$ and endowed with the trivial connection. 

Consider the space
$ {\mathcal{H}} ( S, E, \delta, \varphi )$ which consists of
the holomorphic sections $\Psi$ of $ L_S \otimes \delta$ with respect to the
polarization $E$ such that
$$  \int_S | \Psi (x) |^2 \mu ( x) < \infty  
$$
where $\mu$ the Liouville measure of $S$. 
Here $| \Psi (x) |$ denote
the punctual norm in $L_S \otimes \delta$ . That a section $\Psi$ is
holomorphic with respect to $E$ means that its covariant derivative
with respect to any vector of $\overline E$ vanishes. 

$ {\mathcal{H}} ( S, E, \delta, \varphi )$ is an abstract presentation of the Bargmann space. It
is a Hilbert space with the scalar product
$ \int_S ( \Psi, \Psi')(x) \; \mu (x) $. 
\subsection{Unitary map} \label{sec:unitary-map}

Consider symplectic vector spaces $(S_a, E_a)$ and $(S_b, E_b)$
with positive  polarizations. Let $g$ be a linear symplectomorphism from
$S_a$ to $S_b$. 

\begin{lem} \label{lem:phase}
There exists a unique quadratic
function $ \Phi:  S_b \times S_a \rightarrow \C$ vanishing on the graph
of $g$ and such that $\exp (\Phi)$ is a holomorphic section of $L_{S_b}
\boxtimes \overline{L}_{S_a}$ with respect to $E_b \times
\overline{E}_a$. 
\end{lem}
  
Consider now a half-form line $(\delta_i , \varphi_i)$ of $E_i$ for $i =
a, b$. 
Then for any morphism $(g, \Psi)$ from $a=(S_a, E_a ,\delta_a,
\varphi_a)$ to $b=(S_b, E_b ,\delta_b,
\varphi_b)$ we define a map from $\Hilb (a) $ to $\Hilb (b)$ by 
$$ \bigl( U(g, \Psi ) f \bigr)  (x) = (2 \pi ) ^{-n} \int_{S_a}  \exp ( \Phi (x, y) )
\Psi( f(y)) \mu_b (y) 
$$
where $\mu_b$ is the Liouville measure of $S_b$. 

\begin{theo} \label{theo:foncteur}
For any morphism $(g, \Psi)$, the operator $ U(g, \Psi )$ is unitary.  
Furthermore the map sending $(S, E, \delta, \varphi)$ to $\Hilb (S,
  E, \delta, \varphi)$ and $(g, \Psi)$ to $U(g, \Psi)$ is a functor
  from the category $\demi$ to the category of Hilbert spaces.
\end{theo}

The elementary but long proof of this result will be provided somewhere
else. Applying the functor to the automorphism group of a symplectic
space with polarization and half-form, we obtain the well-known
metaplectic representation. 

\bibliography{biblio}

\end{document}

%% file: index.pstex_t
\begin{picture}(0,0)%
\includegraphics{index.pstex}%
\end{picture}%
\setlength{\unitlength}{3158sp}%
\begingroup\makeatletter\ifx\SetFigFontNFSS\undefined%
\gdef\SetFigFontNFSS#1#2#3#4#5{%
  \reset@font\fontsize{#1}{#2pt}%
  \fontfamily{#3}\fontseries{#4}\fontshape{#5}%
  \selectfont}%
\fi\endgroup%
\begin{picture}(5430,1512)(886,-2389)
\put(1276,-2311){\makebox(0,0)[lb]{\smash{{\SetFigFontNFSS{10}{12.0}{\rmdefault}{\mddefault}{\updefault}{\color[rgb]{0,0,0}0}%
}}}}
\put(2326,-2311){\makebox(0,0)[lb]{\smash{{\SetFigFontNFSS{10}{12.0}{\rmdefault}{\mddefault}{\updefault}{\color[rgb]{0,0,0}$\pi$}%
}}}}
\put(3526,-2311){\makebox(0,0)[lb]{\smash{{\SetFigFontNFSS{10}{12.0}{\rmdefault}{\mddefault}{\updefault}{\color[rgb]{0,0,0}$2\pi$}%
}}}}
\put(4726,-2311){\makebox(0,0)[lb]{\smash{{\SetFigFontNFSS{10}{12.0}{\rmdefault}{\mddefault}{\updefault}{\color[rgb]{0,0,0}$3\pi$}%
}}}}
\put(901,-2161){\makebox(0,0)[lb]{\smash{{\SetFigFontNFSS{10}{12.0}{\rmdefault}{\mddefault}{\updefault}{\color[rgb]{0,0,0}0}%
}}}}
\put(901,-1336){\makebox(0,0)[lb]{\smash{{\SetFigFontNFSS{10}{12.0}{\rmdefault}{\mddefault}{\updefault}{\color[rgb]{0,0,0}1}%
}}}}
\put(1276,-1036){\makebox(0,0)[lb]{\smash{{\SetFigFontNFSS{10}{12.0}{\rmdefault}{\mddefault}{\updefault}{\color[rgb]{0,0,0}$u^2 + v^2$}%
}}}}
\put(6301,-2086){\makebox(0,0)[lb]{\smash{{\SetFigFontNFSS{10}{12.0}{\rmdefault}{\mddefault}{\updefault}{\color[rgb]{0,0,0}$\theta$}%
}}}}
\put(1426,-1711){\makebox(0,0)[lb]{\smash{{\SetFigFontNFSS{10}{12.0}{\rmdefault}{\mddefault}{\updefault}{\color[rgb]{0,0,0}0}%
}}}}
\put(2326,-1711){\makebox(0,0)[lb]{\smash{{\SetFigFontNFSS{10}{12.0}{\rmdefault}{\mddefault}{\updefault}{\color[rgb]{0,0,0}1}%
}}}}
\put(3526,-1711){\makebox(0,0)[lb]{\smash{{\SetFigFontNFSS{10}{12.0}{\rmdefault}{\mddefault}{\updefault}{\color[rgb]{0,0,0}2}%
}}}}
\put(4801,-1711){\makebox(0,0)[lb]{\smash{{\SetFigFontNFSS{10}{12.0}{\rmdefault}{\mddefault}{\updefault}{\color[rgb]{0,0,0}3}%
}}}}
\put(5851,-1711){\makebox(0,0)[lb]{\smash{{\SetFigFontNFSS{10}{12.0}{\rmdefault}{\mddefault}{\updefault}{\color[rgb]{0,0,0}0}%
}}}}
\put(5851,-2311){\makebox(0,0)[lb]{\smash{{\SetFigFontNFSS{10}{12.0}{\rmdefault}{\mddefault}{\updefault}{\color[rgb]{0,0,0}$4\pi$}%
}}}}
\end{picture}%